\documentclass[fleqn,12pt,twoside]{article}

\usepackage[headings]{espcrc1}
\readRCS
$Id: espcrc1.tex,v 1.2 2004/02/24 11:22:11 spepping Exp $
\usepackage{graphicx}
\usepackage[figuresright]{rotating}

\newtheorem{theorem}{Theorem}

\newtheorem{corollary}[theorem]{Corollary}

\newtheorem{lemma}[theorem]{Lemma}

\newenvironment{proof}[1][Proof.]{\begin{trivlist}
\item[\hskip \labelsep {\bfseries #1}]}{\end{trivlist}}

\newcommand{\AmS}{{\protect\the\textfont2
  A\kern-.1667em\lower.5ex\hbox{M}\kern-.125emS}}

\usepackage{amsmath}
\usepackage{amsfonts}
\usepackage{amssymb}
\title{Interval edge-colorings of Cartesian products of graphs I}

\author{Petros A. Petrosyan\address[MCSD]{Department of Informatics and Applied Mathematics,\\
Yerevan State University, 0025, Armenia}%
\address{Institute for Informatics and Automation Problems,\\
National Academy of Sciences, 0014, Armenia}%
\thanks{email: pet\_petros@ipia.sci.am},
        Hrant H. Khachatrian\addressmark[MCSD]%
\thanks {email: hrant@egern.net},
        Hovhannes G. Tananyan\address{Department of Applied Mathematics and Informatics,\\
Russian-Armenian State University, 0051, Armenia}%
\thanks {email: HTananyan@yahoo.com}}

\runtitle{Interval edge-colorings of Cartesian products of graphs I}
\runauthor{Petros A. Petrosyan, Hrant H. Khachatrian and  Hovhannes
G. Tananyan}

\begin{document}

\maketitle

\begin{abstract}
An edge-coloring of a graph $G$ with colors $1,\ldots ,t$ is an
interval $t$-coloring if all colors are used, and the colors of
edges incident to each vertex of $G$ are distinct and form an
interval of integers. A graph $G$ is interval colorable if $G$ has
an interval $t$-coloring for some positive integer $t$. Let
$\mathfrak{N}$ be the set of all interval colorable graphs. For a
graph $G\in \mathfrak{N}$, the least and the greatest values of $t$
for which $G$ has an interval $t$-coloring are denoted by $w(G)$ and
$W(G)$, respectively. In this paper we first show that if $G$ is an
$r$-regular graph and $G\in \mathfrak{N}$, then $W(G\square
P_{m})\geq W(G)+W(P_{m})+(m-1)r$ ($m\in \mathbb{N}$) and $W(G\square
C_{2n})\geq W(G)+W(C_{2n})+nr$ ($n\geq 2$). Next, we investigate
interval edge-colorings of grids, cylinders and tori. In particular,
we prove that if $G\square H$ is planar and both factors have at
least $3$ vertices, then $G\square H\in \mathfrak{N}$ and
$w(G\square H)\leq 6$. Finally, we confirm the first author's
conjecture on the $n$-dimensional cube $Q_{n}$ and show that $Q_{n}$
has an interval $t$-coloring if and only if $n\leq t\leq
\frac{n\left(n+1\right)}{2}$.\\

Keywords: edge-coloring, interval coloring, grid, cylinder, torus, $n$-dimensional cube\\

AMS Subject Classification: 05C15, 05C76\\
\end{abstract}

\section{Introduction}\

An edge-coloring of a graph $G$ with colors $1,\ldots,t$ is an
interval $t$-coloring if all colors are used, and the colors of
edges incident to each vertex of $G$ are distinct and form an
interval of integers. A graph $G$ is interval colorable if $G$ has
an interval $t$-coloring for some positive integer $t$. Let
$\mathfrak{N}$ be the set of all interval colorable graphs
\cite{b1,b14}. For a graph $G\in \mathfrak{N}$, the least and the
greatest values of $t$ for which $G$ has an interval $t$-coloring
are denoted by $w(G)$ and $W(G)$, respectively. The concept of
interval edge-coloring was introduced by Asratian and Kamalian
\cite{b1}. In \cite{b1}, they proved the following:

\begin{theorem}
\label{mytheorem1} If $G$ is a regular graph, then

\begin{description}
\item[(1)] $G\in \mathfrak{N}$ if and only if $\chi^{\prime}(G)=\Delta(G)$.

\item[(2)] If $G\in \mathfrak{N}$ and $w(G)\leq t\leq W(G)$, then $G$
has an interval $t$-coloring.
\end{description}
\end{theorem}

In \cite{b2}, Asratian and Kamalian investigated interval
edge-colorings of connected graphs. In particular, they obtained the
following two results.

\begin{theorem}
\label{mytheorem2} If $G$ is a connected graph and $G\in
\mathfrak{N}$, then
\begin{center}
$W(G)\leq \left(\mathrm{diam}(G)+1\right)\left(\Delta(G) -1\right)
+1$.
\end{center}
\end{theorem}

\begin{theorem}
\label{mytheorem3} If $G$ is a connected bipartite graph and $G\in
\mathfrak{N}$, then
\begin{center}
$W(G)\leq \mathrm{diam}(G)\left(\Delta(G) -1\right) +1$.
\end{center}
\end{theorem}

Recently, Kamalian and Petrosyan \cite{b16} showed that these upper
bounds cannot be significantly improved.

In \cite{b13}, Kamalian investigated interval colorings of complete
bipartite graphs and trees. In particular, he proved the following:

\begin{theorem}
\label{mytheorem4} For any $r,s\in \mathbb{N}$, the complete
bipartite graph $K_{r,s}$ is interval colorable, and

\begin{description}
\item[(1)] $w\left(K_{r,s}\right)=r+s-\gcd(r,s)$,

\item[(2)] $W\left(K_{r,s}\right)=r+s-1$,

\item[(3)] if $w\left(K_{r,s}\right)\leq t\leq W\left(K_{r,s}\right)$, then $K_{r,s}$
has an interval $t$-coloring.
\end{description}
\end{theorem}

In \cite{b21}, Petrosyan investigated interval colorings of complete
graphs and $n$-dimensional cubes. In particular, he obtained the
following two results.

\begin{theorem}
\label{mytheorem5} If $n=p2^{q}$, where $p$ is odd and $q$ is
nonnegative, then
\begin{center}
$W\left(K_{2n}\right)\geq 4n-2-p-q$.
\end{center}
\end{theorem}

\begin{theorem}
\label{mytheorem6} $W\left(Q_{n}\right)\geq \frac{n(n+1)}{2}$ for
any $n\in\mathbb{N}$.
\end{theorem}

The $NP$-completeness of the problem of the existence of an interval
edge-coloring of an arbitrary bipartite graph was shown in
\cite{b24}. A similar result for regular graphs was obtained in
\cite{b1,b2}. In \cite{b19,b22,b23}, interval edge-colorings of
various products of graphs were investigated. Some interesting
results on interval colorings were also obtained in
\cite{b3,b4,b6,b7,b8,b9,b10,b14,b15,b16,b17,b18,b19,b20}. Surveys on
this topic can be found in some books \cite{b3,b12,b19}.

In this paper we focus only on interval edge-colorings of Cartesian
products of graphs.

\bigskip

\section{Notations, definitions and auxiliary results}\

Throughout this paper all graphs are finite, undirected, and have no
loops or multiple edges. Let $V(G)$ and $E(G)$ denote the sets of
vertices and edges of $G$, respectively. The degree of a vertex $v$
in $G$ is denoted by $d_{G}(v)$, the maximum degree of $G$ by
$\Delta (G)$, and the chromatic index of $G$ by $\chi^{\prime }(G)$.
If $G$ is a connected graph, then the distance between two vertices
$u$ and $v$ in $G$, we denote by $d(u,v)$, and the diameter of $G$
by $\mathrm{diam}(G)$. We use the standard notations $P_{n}$,
$C_{n}$, $K_{n}$ and $Q_{n}$ for the simple path, simple cycle,
complete graph on vertices and the $n$-dimensional cube,
respectively. A partial edge-coloring of a graph $G$ is a coloring
of some of the edges of $G$ such that no two adjacent edges receive
the same color. If $\alpha$ is a partial edge-coloring of $G$ and
$v\in V(G)$, then $S\left(v,\alpha\right)$ denotes the set of colors
appearing on colored edges incident to $v$.

Let $[t]$ denote the set of the first $t$ natural numbers. Let
$\left\lfloor a\right\rfloor $ ($\left\lceil a\right\rceil$) denote
the largest (least) integer less (greater) than or equal to $a$. For
two positive integers $a$ and $b$ with $a\leq b$, the set
$\left\{a,\ldots ,b\right\}$ is denoted by $\left[a,b\right]$. The
terms and concepts that we do not define can be found in \cite{b25}.

Let $G$ and $H$ be graphs. The Cartesian product $G\square H$ is
defined as follows:
\begin{center}
$V(G\square H)=V(G)\times V(H)$,
\end{center}
\begin{center}
$E(G\square H)=\{(u_{1},v_{1})(u_{2},v_{2})\colon\,
u_{1}=u_{2}~and~v_{1}v_{2}\in E(H)~or~v_{1}=v_{2}~and~ u_{1}u_{2}\in
E(G)\}$.
\end{center}

Clearly, if $G$ and $H$ are connected graphs, then $G\square H$ is
connected, too. Moreover, $\Delta(G\square H)=\Delta(G)+\Delta(H)$
and $\mathrm{diam}(G\square H)=\mathrm{diam}(G)+\mathrm{diam}(H)$.
The $k$-dimensional grid $G(n_{1},n_{2},\ldots,n_{k})$, $n_{i}\in
\mathbb{N}$ is the Cartesian product of paths $P_{n_{1}}\square
P_{n_{2}}\square\cdots\square P_{n_{k}}$. The cylinder
$C(n_{1},n_{2})$ is the Cartesian product $P_{n_{1}}\square
C_{n_{2}}$, and the torus $T(n_{1},n_{2})$ is the Cartesian product
of cycles $C_{n_{1}}\square C_{n_{2}}$.\\

We also need the following two lemmas.

\begin{lemma}
\label{mylemma1} If $\alpha$ is an edge-coloring of a connected
graph $G$ with colors $1,\ldots,t$ such that the edges incident to
each vertex $v\in V(G)$ are colored by distinct and consecutive
colors, and $\min_{e\in E(G)}\{\alpha(e)\}=1$, $\max_{e\in
E(G)}\{\alpha(e)\}=t$, then $\alpha$ is an interval $t$-coloring of
$G$.
\end{lemma}
\begin{proof} For the proof of the lemma, it suffices to show that all
colors are used in the coloring $\alpha$ of $G$.

Let $u$ and $w$ be vertices such that $1\in S(u,\alpha)$ and $t\in
S(w,\alpha)$. Also, let $P=v_{1},\ldots,v_{k}$, where $u=v_{1}$ and
$v_{k}=w$ be a $u,w$-path in $G$. If $k=1$, then $t\in S(u,\alpha)$
and all colors appear on edges incident to $u$. Assume that $k\geq
2$. The sets $S(v_{i},\alpha)$ for $v_{i}\in V(P)$ are intervals,
and for $2\leq i\leq k$, intervals $S(v_{i-1},\alpha)$ and
$S(v_{i},\alpha)$ share a color. Thus, the sets
$S(v_{1},\alpha),\ldots,S(v_{k},\alpha)$ cover $[1,t]$.~$\square$
\end{proof}

The next lemma was proved by Behzad and Mahmoodian in \cite{b5}.

\begin{lemma}
\label{mylemma2} If both $G$ and $H$ have at least $3$ vertices,
then the Cartesian product $G\square H$ is planar if and only if
$G\square H=G(m,n)$ or $G\square H=C(m,n)$.
\end{lemma}
\bigskip

\section{The Cartesian product of regular graphs}\

Interval edge-colorings of Cartesian products of graphs were first
investigated by Giaro and Kubale in \cite{b7}, where they proved the
following:

\begin{theorem}
\label{mytheorem7} If $G\in \mathfrak{N}$, then $G\square P_{m}\in
\mathfrak{N}$ $(m\in \mathbb{N})$ and $G\square C_{2n}\in
\mathfrak{N}$ $(n\geq 2)$.
\end{theorem}

It is well-known that $P_{m},C_{2n}\in \mathfrak{N}$ and
$W\left(P_{m}\right)=m-1$, $W\left(C_{2n}\right)=n+1$ for $m\in
\mathbb{N}$ and $n\geq 2$. Later, Giaro and Kubale \cite{b9,b19}
proved a more general result.

\begin{theorem}
\label{mytheorem8} If $G,H\in \mathfrak{N}$, then $G\square H\in
\mathfrak{N}$. Moreover, $w(G\square H)\leq w(G)+w(H)$ and
$W(G\square H)\geq W(G)+W(H)$.
\end{theorem}

Let us note that if $G\in \mathfrak{N}$ and $H=P_{m}$ or $H=C_{2n}$,
then, by Theorem \ref{mytheorem8}, we obtain $w(G\square H)\leq
w(G)+2$ and $W(G\square P_{m})\geq W(G)+m-1$, $W(G\square
C_{2n})\geq W(G)+n+1$. Now we improve the lower bound in Theorem
\ref{mytheorem8} for $W\left(G\square P_{m}\right)$ and
$W\left(G\square C_{2n}\right)$ when $G$ is a regular graph and
$G\in \mathfrak{N}$. More precisely, we show that the following two
theorems hold.

\begin{theorem}
\label{mytheorem9} If $G$ is an $r$-regular graph and $G\in
\mathfrak{N}$, then $G\square P_{m}\in \mathfrak{N}$ $(m\in
\mathbb{N})$ and $W\left(G\square P_{m}\right)\geq
W(G)+W\left(P_{m}\right)+(m-1)r$.
\end{theorem}
\begin{proof} For the proof, we construct an edge-coloring of the graph $G\square P_{m}$ that satisfies the
specified conditions.

Let $V(G)=\left\{v_{1},\ldots,v_{n}\right\}$ and $V\left(G\square
P_{m}\right)=\bigcup_{i=1}^{m}V^{i}$, where
$V^{i}=\left\{v_{j}^{(i)}\colon\,1\leq j\leq n\right\}$. Also, let
$E\left(G\square P_{m}\right)=\bigcup_{i=1}^{m}E^{i}\cup
\bigcup_{j=1}^{n}E_{j}$, where
\begin{center}
$E^{i}=\left\{v_{j}^{(i)}v_{k}^{(i)}\colon\,v_{j}v_{k}\in E(G)
\right\}$ and $E_{j}=\left\{v_{j}^{(i)}v_{j}^{(i+1)}\colon\,1\leq
i\leq m-1\right\}$.
\end{center}

For $1\leq i\leq m$, define a subgraph $G^{i}$ of the graph
$G\square P_{m}$ as follows: $G^{i}=\left(V^{i},E^{i}\right)$.
Clearly, $G^{i}$ is isomorphic to $G$ for $1\leq i\leq m$. Since
$G\in \mathfrak{N}$, there exists an interval $W(G)$-coloring
$\alpha$ of $G$. Now we define an edge-coloring $\beta$ of the
subgraphs $G^{1},\ldots,G^{m}$.

For $1\leq i\leq m$ and for every edge $v_{j}^{(i)}v_{k}^{(i)}\in
E(G^{i})$, let
\begin{center}
$\beta \left(v_{j}^{(i)}v_{k}^{(i)}\right)=\alpha
(v_{j}v_{k})+(i-1)(r+1)$.
\end{center}

It is easy to see that the color of each edge of the subgraph
$G^{i}$ is obtained by shifting the color of the associated edge of
$G$ by $(i-1)(r+1)$, thus the set $S\left(v_{j}^{(i)},\beta\right)$
is an interval for each vertex $v_{j}^{(i)}\in V(G^{i})$, where
$1\leq i\leq m$, $1\leq j\leq n$. Now we define an edge-coloring
$\gamma$ of the graph $G\square P_{m}$.

For every $e\in E\left(G\square P_{m}\right)$, let
\begin{center}
$\gamma(e)= \left\{
\begin{tabular}{ll}
$\beta(e)$, & if $e\in E(G^{i}),$\\
$\max S\left(v_{j}^{(i)},\beta\right)+1$, & if $e=v_{j}^{(i)}v_{j}^{(i+1)}\in E_{j}$,\\
\end{tabular}%
\right.$
\end{center}
where $1\leq i\leq m,1\leq j\leq n$.

Let us prove that $\gamma$ is an interval
$\left(W(G)+W\left(P_{m}\right)+(m-1)r\right)$-coloring of the graph
$G\square P_{m}$ for $m\in \mathbb{N}$.

First we show that the set $S\left(v_{j}^{(i)},\gamma\right)$ is an
interval for each vertex $v_{j}^{(i)}\in V\left(G\square
P_{m}\right)$, where $1\leq i\leq m,1\leq j\leq n$.

Case 1: $i=1$, $1\leq j\leq n$.

By the definition of $\gamma$ and taking into account that $\max
S\left(v_{j},\alpha\right)-\min S\left(v_{j},\alpha\right)=r-1$ for
$1\leq j\leq n$, we have

\begin{eqnarray*}
S\left(v_{j}^{(1)},\gamma\right) &=&\left\{\min
S\left(v_{j},\alpha\right),\ldots,\max
S\left(v_{j},\alpha\right)\right\}\cup \left\{\max S\left(v_{j},\alpha\right)+1\right\}\\
&=& [\min S\left(v_{j},\alpha\right),\max
S\left(v_{j},\alpha\right)+1].
\end{eqnarray*}

Case 2: $2\leq i\leq m-1$, $1\leq j\leq n$.

By the definition of $\gamma$ and taking into account that $\max
S\left(v_{j},\alpha\right)-\min S\left(v_{j},\alpha\right)=r-1$ for
$1\leq j\leq n$, we have

\begin{eqnarray*}
S\left(v_{j}^{(i)},\gamma\right) &=&\left\{\min
S\left(v_{j},\alpha\right)+(i-1)(r+1),\ldots,\max
S\left(v_{j},\alpha\right)+(i-1)(r+1)\right\}\\
&&\cup \left\{\max S\left(v_{j},\alpha\right)+(i-2)(r+1)+1\right\}\cup \left\{\max S\left(v_{j},\alpha\right)+(i-1)(r+1)+1\right\}\\
&=&\left\{\min S\left(v_{j},\alpha\right)+(i-1)(r+1),\ldots,\max
S\left(v_{j},\alpha\right)+(i-1)(r+1)+1\right\}\\
&&\cup \left\{\max S\left(v_{j},\alpha\right)+(i-2)(r+1)+1\right\}\\
&=&\left\{\min S\left(v_{j},\alpha\right)+(i-1)(r+1)-1,\ldots,\max S\left(v_{j},\alpha\right)+(i-1)(r+1)+1\right\}\\
&=& [\min S\left(v_{j},\alpha\right)+(i-1)(r+1)-1,\max
S\left(v_{j},\alpha\right)+(i-1)(r+1)+1].
\end{eqnarray*}

Case 3: $i=m$, $1\leq j\leq n$.

By the definition of $\gamma$ and taking into account that $\max
S\left(v_{j},\alpha\right)-\min S\left(v_{j},\alpha\right)=r-1$ for
$1\leq j\leq n$, we have

\begin{eqnarray*}
S\left(v_{j}^{(m)},\gamma\right) &=&\left\{\min
S\left(v_{j},\alpha\right)+(m-1)(r+1),\ldots,\max
S\left(v_{j},\alpha\right)+(m-1)(r+1)\right\}\\
&&\cup \left\{\max S\left(v_{j},\alpha\right)+(m-2)(r+1)+1\right\}\\
&=& \left\{\min S\left(v_{j},\alpha\right)+(m-1)(r+1)-1,\ldots,\max
S\left(v_{j},\alpha\right)+(m-1)(r+1)\right\}\\
&=& [\min S\left(v_{j},\alpha\right)+(m-1)(r+1)-1,\max
S\left(v_{j},\alpha\right)+(m-1)(r+1)].
\end{eqnarray*}

Next we prove that in the coloring $\gamma$ all colors are used.
Clearly, there exists an edge $v_{j_{0}}^{(1)}v_{k_{0}}^{(1)}\in
E(G^{1})$ such that
$\gamma\left(v_{j_{0}}^{(1)}v_{k_{0}}^{(1)}\right)=1$, since in the
coloring $\alpha$ there exists an edge $v_{j_{0}}v_{k_{0}}$ with
$\alpha\left(v_{j_{0}}v_{k_{0}}\right)=1$ and
$\gamma\left(v_{j_{0}}^{(1)}v_{k_{0}}^{(1)}\right)=
\beta\left(v_{j_{0}}^{(1)}v_{k_{0}}^{(1)}\right)=
\alpha\left(v_{j_{0}}v_{k_{0}}\right)$. Similarly, there exists an
edge $v_{j_{1}}^{(m)}v_{k_{1}}^{(m)}\in E(G^{m})$ such that
$\gamma\left(v_{j_{1}}^{(m)}v_{k_{1}}^{(m)}\right)=W(G)+(m-1)(r+1)=W(G)+W\left(P_{m}\right)+(m-1)r$,
since in the coloring $\alpha$ there exists an edge
$v_{j_{1}}v_{k_{1}}$ with
$\alpha\left(v_{j_{1}}v_{k_{1}}\right)=W(G)$ and
$\gamma\left(v_{j_{1}}^{(m)}v_{k_{1}}^{(m)}\right)=
\beta\left(v_{j_{1}}^{(m)}v_{k_{1}}^{(m)}\right)=
\alpha\left(v_{j_{1}}v_{k_{1}}\right)+(m-1)(r+1)$. Now, by Lemma
\ref{mylemma1}, we have that for each $t\in
[W(G)+W\left(P_{m}\right)+(m-1)r]$, there is an edge $e\in
E\left(G\square P_{m}\right)$ with $\gamma(e)=t$.

This shows that $\gamma$ is an interval
$\left(W(G)+W\left(P_{m}\right)+(m-1)r\right)$-coloring of the graph
$G\square P_{m}$ for $m\in \mathbb{N}$. Thus, $G\square P_{m}\in
\mathfrak{N}$ and $W\left(G\square P_{m}\right)\geq
W(G)+W\left(P_{m}\right)+(m-1)r$. ~$\square$
\end{proof}

\begin{corollary}
\label{mycorollary1} If $G$ is an $r$-regular graph and $G\in
\mathfrak{N}$, then $G\square Q_{n}\in \mathfrak{N}$
$(n\in\mathbb{N})$ and
\begin{center}
$W\left(G\square Q_{n}\right)\geq W(G)+\frac{n(n+2r+1)}{2}$.
\end{center}
\end{corollary}

\begin{theorem}
\label{mytheorem10} If $G$ is an $r$-regular graph and $G\in
\mathfrak{N}$, then $G\square C_{2n}\in \mathfrak{N}$ $(n\geq 2)$
and $W\left(G\square C_{2n}\right)\geq
W(G)+W\left(C_{2n}\right)+nr$.
\end{theorem}
\begin{proof} For the proof, we construct an edge-coloring of the
graph $G\square C_{2n}$ that satisfies the specified conditions.

Let $V(G)=\left\{v_{1},\ldots,v_{p}\right\}$ and $V\left(G\square
C_{2n}\right)=\bigcup_{i=1}^{2n}V^{i}$, where
$V^{i}=\left\{v_{j}^{(i)}\colon\,1\leq j\leq p\right\}$. Also, let
$E\left(G\square C_{2n}\right)=\bigcup_{i=1}^{2n}E^{i}\cup
\bigcup_{j=1}^{p}E_{j}$, where
\begin{center}
$E^{i}=\left\{v_{j}^{(i)}v_{k}^{(i)}\colon\,v_{j}v_{k}\in E(G)
\right\}$ and $E_{j}=\left\{v_{j}^{(i)}v_{j}^{(i+1)}\colon\,1\leq
i\leq 2n-1\right\}\cup \left\{v_{j}^{(1)}v_{j}^{(2n)}\right\}$.
\end{center}

For $1\leq i\leq 2n$, define a subgraph $G^{i}$ of the graph
$G\square C_{2n}$ as follows: $G^{i}=\left(V^{i},E^{i}\right)$.
Clearly, $G^{i}$ is isomorphic to $G$ for $1\leq i\leq 2n$. Since
$G\in \mathfrak{N}$, there exists an interval $W(G)$-coloring
$\alpha$ of $G$. Now we define an edge-coloring $\beta$ of the
subgraphs $G^{1},\ldots,G^{2n}$.

For $1\leq i\leq 2n$ and for every edge $v_{j}^{(i)}v_{k}^{(i)}\in
E(G^{i})$, let
\begin{center}
$\beta\left(v_{j}^{(i)}v_{k}^{(i)}\right)= \left\{
\begin{tabular}{ll}
$\alpha\left(v_{j}v_{k}\right)$, & if $i=1$,\\
$\alpha\left(v_{j}v_{k}\right)+(i-1)(r+1)+1$, & if $2\leq i\leq n+1$,\\
$\alpha\left(v_{j}v_{k}\right)+(2n+1-i)(r+1)$, & if $n+2\leq i\leq 2n$.\\
\end{tabular}%
\right.$
\end{center}

It is easy to see that the color of each edge of the subgraph
$G^{i}$ is obtained by shifting the color of the associated edge of
$G$ by $(i-1)(r+1)+1$ for $2\leq i\leq n+1$, and by $(2n-i+1)(r+1)$
for $n+2\leq i\leq 2n$, thus the set
$S\left(v_{j}^{(i)},\beta\right)$ is an interval for each vertex
$v_{j}^{(i)}\in V(G^{i})$, where $1\leq i\leq 2n$, $1\leq j\leq p$.
Now we define an edge-coloring $\gamma$ of the graph $G\square
C_{2n}$.

For every $e\in E\left(G\square C_{2n}\right)$, let
\begin{center}
$\gamma(e)= \left\{
\begin{tabular}{ll}
$\beta(e)$, & if $e\in E(G^{i}),$\\
$\max S\left(v_{j}^{(1)},\beta\right)+1$, & if $e=v_{j}^{(1)}v_{j}^{(2n)}\in E_{j}$,\\
$\max S\left(v_{j}^{(1)},\beta\right)+2$, & if $e=v_{j}^{(1)}v_{j}^{(2)}\in E_{j}$,\\
$\max S\left(v_{j}^{(i)},\beta\right)+1$, & if $e=v_{j}^{(i)}v_{j}^{(i+1)}\in E_{j},2\leq i\leq n$,\\
$\max S\left(v_{j}^{(i)},\beta\right)+1$, & if $e=v_{j}^{(i-1)}v_{j}^{(i)}\in E_{j},n+2\leq i\leq 2n$,\\
\end{tabular}%
\right.$
\end{center}
where $1\leq i\leq 2n,1\leq j\leq p$.

Let us prove that $\gamma$ is an interval
$\left(W(G)+W\left(C_{2n}\right)+nr\right)$-coloring of the graph
$G\square C_{2n}$ for $n\geq 2$.

First we show that the set $S\left(v_{j}^{(i)},\gamma\right)$ is an
interval for each vertex $v_{j}^{(i)}\in V\left(G\square
C_{2n}\right)$, where $1\leq i\leq 2n,1\leq j\leq p$.

Case 1: $i=1$, $1\leq j\leq p$.

By the definition of $\gamma$ and taking into account that $\max
S\left(v_{j},\alpha\right)-\min S\left(v_{j},\alpha\right)=r-1$ for
$1\leq j\leq p$, we have

\begin{eqnarray*}
S\left(v_{j}^{(1)},\gamma\right) &=&\left\{\min
S\left(v_{j},\alpha\right),\ldots,\max
S\left(v_{j},\alpha\right)\right\}\cup \left\{\max S\left(v_{j},\alpha\right)+2\right\}\cup \left\{\max S\left(v_{j},\alpha\right)+1\right\}\\
&=& [\min S\left(v_{j},\alpha\right),\max
S\left(v_{j},\alpha\right)+2].
\end{eqnarray*}

Case 2: $2\leq i\leq n$, $1\leq j\leq p$.

By the definition of $\gamma$ and taking into account that $\max
S\left(v_{j},\alpha\right)-\min S\left(v_{j},\alpha\right)=r-1$ for
$1\leq j\leq p$, we have

\begin{eqnarray*}
S\left(v_{j}^{(i)},\gamma\right) &=&\left\{\min
S\left(v_{j},\alpha\right)+(i-1)(r+1)+1,\ldots,\max
S\left(v_{j},\alpha\right)+(i-1)(r+1)+1\right\}\\
&&\cup \left\{\max S\left(v_{j},\alpha\right)+(i-2)(r+1)+2\right\}\cup \left\{\max S\left(v_{j},\alpha\right)+(i-1)(r+1)+2\right\}\\
&=& [\min S\left(v_{j},\alpha\right)+(i-1)(r+1),\max
S\left(v_{j},\alpha\right)+(i-1)(r+1)+2].
\end{eqnarray*}

Case 3: $i=n+1$, $1\leq j\leq p$.

By the definition of $\gamma$ and taking into account that $\max
S\left(v_{j},\alpha\right)-\min S\left(v_{j},\alpha\right)=r-1$ for
$1\leq j\leq p$, we have

\begin{eqnarray*}
S\left(v_{j}^{(n+1)},\gamma\right) &=&\left\{\min
S\left(v_{j},\alpha\right)+n(r+1)+1,\ldots,\max
S\left(v_{j},\alpha\right)+n(r+1)+1\right\}\\
&&\cup \left\{\max S\left(v_{j},\alpha\right)+(n-1)(r+1)+2\right\}\cup \left\{\max S\left(v_{j},\alpha\right)+(n-1)(r+1)+1\right\}\\
&=& [\min S\left(v_{j},\alpha\right)+n(r+1)-1,\max
S\left(v_{j},\alpha\right)+n(r+1)+1].
\end{eqnarray*}

Case 4: $n+2\leq i\leq 2n$, $1\leq j\leq p$.

By the definition of $\gamma$ and taking into account that $\max
S\left(v_{j},\alpha\right)-\min S\left(v_{j},\alpha\right)=r-1$ for
$1\leq j\leq p$, we have

\begin{eqnarray*}
S\left(v_{j}^{(i)},\gamma\right) &=&\left\{\min
S\left(v_{j},\alpha\right)+(2n+1-i)(r+1),\ldots,\max
S\left(v_{j},\alpha\right)+(2n+1-i)(r+1)\right\}\\
&&\cup \left\{\max S\left(v_{j},\alpha\right)+(2n+1-i)(r+1)+1\right\}\cup \left\{\max S\left(v_{j},\alpha\right)+(2n-i)(r+1)+1\right\}\\
&=& [\min S\left(v_{j},\alpha\right)+(2n-i+1)(r+1)-1,\max
S\left(v_{j},\alpha\right)+(2n-i+1)(r+1)+1].
\end{eqnarray*}

Next we prove that in the coloring $\gamma$ all colors are used.
Clearly, there exists an edge $v_{j_{0}}^{(1)}v_{k_{0}}^{(1)}\in
E(G^{1})$ such that
$\gamma\left(v_{j_{0}}^{(1)}v_{k_{0}}^{(1)}\right)=1$, since in the
coloring $\alpha$ there exists an edge $v_{j_{0}}v_{k_{0}}$ with
$\alpha\left(v_{j_{0}}v_{k_{0}}\right)=1$ and
$\gamma\left(v_{j_{0}}^{(1)}v_{k_{0}}^{(1)}\right)=
\beta\left(v_{j_{0}}^{(1)}v_{k_{0}}^{(1)}\right)=
\alpha\left(v_{j_{0}}v_{k_{0}}\right)$. Similarly, there exists an
edge $v_{j_{1}}^{(n+1)}v_{k_{1}}^{(n+1)}\in E(G^{n+1})$ such that
$\gamma\left(v_{j_{1}}^{(n+1)}v_{k_{1}}^{(n+1)}\right)=W(G)+n(r+1)+1=W(G)+W\left(C_{2n}\right)+nr$,
since in the coloring $\alpha$ there exists an edge
$v_{j_{1}}v_{k_{1}}$ with
$\alpha\left(v_{j_{1}}v_{k_{1}}\right)=W(G)$ and
$\gamma\left(v_{j_{1}}^{(n+1)}v_{k_{1}}^{(n+1)}\right)=
\beta\left(v_{j_{1}}^{(n+1)}v_{k_{1}}^{(n+1)}\right)=
\alpha\left(v_{j_{1}}v_{k_{1}}\right)+n(r+1)+1$. Now, by Lemma
\ref{mylemma1}, we have that for each $t\in
[W(G)+W\left(C_{2n}\right)+nr]$, there is an edge $e\in
E\left(G\square C_{2n}\right)$ with $\gamma(e)=t$.

This shows that $\gamma$ is an interval
$\left(W(G)+W\left(C_{2n}\right)+nr\right)$-coloring of the graph
$G\square C_{2n}$ for $n\geq 2$. Thus, $G\square C_{2n}\in
\mathfrak{N}$ and $W\left(G\square C_{2n}\right)\geq
W(G)+W\left(C_{2n}\right)+nr$. ~$\square$
\end{proof}

From Theorems \ref{mytheorem5} and \ref{mytheorem10}, we have:

\begin{corollary}
\label{mycorollary2} If $n=p2^{q}$, where $p$ is odd and $q$ is
nonnegative, then
\begin{center}
$W\left(K_{2n}\square C_{2n}\right)\geq 2n^{2}+4n-1-p-q$.
\end{center}
\end{corollary}

Note that the lower bound in Corollary \ref{mycorollary2} is close
to the upper bound for $W\left(K_{2n}\square C_{2n}\right)$, since
$\Delta\left(K_{2n}\square C_{2n}\right)=2n+1$ and
$\mathrm{diam}\left(K_{2n}\square C_{2n}\right)=n+1$, by Theorem
\ref{mytheorem2}, we have $W\left(K_{2n}\square C_{2n}\right)\leq
2n^{2}+4n+1$.
\bigskip

\section{Grids, cylinders and tori}\

Interval edge-colorings of grids, cylinders and tori were first
considered by Giaro and Kubale in \cite{b7}, where they proved the
following:

\begin{theorem}\label{mytheorem11} If $G=G(n_{1},n_{2},\ldots,n_{k})$
or $G=C(m,2n)$, $m\in \mathbb{N}, n\geq 2$, or $G=T(2m,2n)$,
$m,n\geq 2$, then $G\in \mathfrak{N}$ and $w(G)=\Delta(G)$.
\end{theorem}

For the greatest possible number of colors in interval colorings of
cylinders and tori, Petrosyan and Karapetyan \cite{b20} proved the
following theorems:

\begin{theorem}
\label{mytheorem12} For any $m\in \mathbb{N}, n\geq 2$, we have
$W(C(m,2n))\geq 3m+n-2$.
\end{theorem}

\begin{theorem}
\label{mytheorem13} For any $m,n \geq 2$, we have $W(T(2m,2n))\geq
\max\{3m+n,3n+m\}$.
\end{theorem}

First we consider grids. It is easy to see that
$W\left(G(2,n)\right)=2n-1$ for any $n\in \mathbb{N}$. Now we
provide a lower bound for $W\left(G(m,n)\right)$ when $m,n\geq 2$.

\begin{theorem}
\label{mytheorem14} For any $m,n\geq 2$, we have $W(G(m,n))\geq
2(m+n-3)$.
\end{theorem}
\begin{proof}For the proof, we are going to construct an edge-coloring
 of the graph $G(m,n)$ that satisfies the specified conditions.

Let $V(G(m,n))=\left\{v_{j}^{(i)}\colon\,1\leq i\leq m,1\leq j\leq
n\right\}$ and $E(G(m,n))=\bigcup_{i=1}^{m}E^{i}\cup
\bigcup_{j=1}^{n}E_{j}$, where
\begin{center}
$E^{i}=\left\{v_{j}^{(i)}v_{j+1}^{(i)}\colon\,1\leq j\leq n-1
\right\}$ and $E_{j}=\left\{v_{j}^{(i)}v_{j}^{(i+1)}\colon\,1\leq
i\leq m-1\right\}$.
\end{center}

Define an edge-coloring $\alpha$ of $G(m,n)$ as follows:\\
\begin{description}
\item[(1)] for $i=1,\ldots,m-1$, $j=1,\ldots,n-1$, let
\begin{center}
$\alpha\left(v_{j}^{(i)}v_{j}^{(i+1)}\right)=2(i+j)-3$;
\end{center}
\item[(2)] for $i=1,\ldots,m-1$, let
\begin{center}
$\alpha\left(v_{n}^{(i)}v_{n}^{(i+1)}\right)=2(n+i)-5$;
\end{center}
\item[(3)] for $j=1,\ldots,n-1$, let
\begin{center}
$\alpha\left(v_{j}^{(1)}v_{j+1}^{(1)}\right)=2j$;
\end{center}
\item[(4)] for $i=2,\ldots,m$, $j=1,\ldots,n-1$, let
\begin{center}
$\alpha\left(v_{j}^{(i)}v_{j+1}^{(i)}\right)=2(i+j)-4$.
\end{center}
\end{description}

It is easy to see that $\alpha$ is an interval $(2(m+n-3))$-coloring
of $G(m,n)$ when $m,n\geq 2$.~$\square$
\end{proof}

Note that the lower bound in Theorem \ref{mytheorem14} is not far
from the upper bound for $W\left(G(m,n)\right)$, since $G(m,n)$ is
bipartite, $2\leq \Delta\left(G(m,n)\right)\leq 4$ and
$\mathrm{diam}\left(G(m,n)\right)=m+n-2$, by Theorem
\ref{mytheorem3}, we have $W\left(G(m,n)\right)\leq 3(m+n-2)+1$.\\

From Theorems \ref{mytheorem8} and \ref{mytheorem14}, we have:

\begin{corollary}
\label{mycorollary3} If $n_{1}\geq\cdots \geq n_{2k}\geq 2$ $(k\in
\mathbb{N})$, then
\begin{center}
$W(G(n_{1},n_{2},\ldots,n_{2k}))\geq 2\sum_{i=1}^{2k}n_{i}-6k$,
\end{center}
and if $n_{1}\geq\cdots \geq n_{2k+1}\geq 2$ $(k\in
\mathbb{N})$, then
\begin{center}
$W(G(n_{1},n_{2},\ldots,n_{2k+1}))\geq
2\sum_{i=1}^{2k}n_{i}+n_{2k+1}-6k-1$.
\end{center}
\end{corollary}

Next we consider cylinders. In \cite{b18}, Khchoyan proved the
following:

\begin{theorem}
\label{mytheorem15} For any $n\geq 3$, we have
\begin{description}
\item[(1)] $C(2,n)\in \mathfrak{N}$,

\item[(2)] $w\left(C(2,n)\right)=3$,

\item[(3)] $W\left(C(2,n)\right)=n+2$,

\item[(4)] if $w\left(C(2,n)\right)\leq t\leq W\left(C(2,n)\right)$, then $C(2,n)$
has an interval $t$-coloring.
\end{description}
\end{theorem}

Now we prove some general results on cylinders.

\begin{theorem}
\label{mytheorem16} For any $m\geq 3, n\in \mathbb{N}$, we have
$C(m,2n+1)\in \mathfrak{N}$ and
\begin{center}
$w\left(C(m,2n+1)\right)= \left\{
\begin{tabular}{ll}
$4$, & if $m$ is even,\\
$6$, & if $m$ is odd.\\
\end{tabular}%
\right.$
\end{center}
\end{theorem}
\begin{proof} Let $V(C(m,2n+1))=\left\{v_{j}^{(i)}\colon\,1\leq i\leq m,1\leq j\leq
2n+1\right\}$ and $E(C(m,2n+1))=\bigcup_{i=1}^{m}{E}^{i}\cup
\bigcup_{j=1}^{2n+1}{E}_{j}$ , where
\begin{center}
$E^{i}=\left\{v_{j}^{(i)}v_{j+1}^{(i)}\colon\,1\leq j\leq
2n\right\}\cup \left\{v_{1}^{(i)}v_{2n+1}^{(i)}\right\}$,
$E_{j}=\left\{v_{j}^{(i)}v_{j}^{(i+1)}\colon\,1\leq i\leq
m-1\right\}$.
\end{center}

First we show that if $m$ is even, then $C(m,2n+1)$ has an interval
$4$-coloring.

For $1\leq i\leq \frac{m}{2}$, define a subgraph $C^{i}$ of the
graph $C(m,2n+1)$ as follows:
\begin{center}
$C^{i}=\left(V^{2i-1}\cup V^{2i},E^{2i-1}\cup E^{2i}\cup
\left\{v_{j}^{(2i-1)}v_{j}^{(2i)}\colon\,1\leq j\leq
2n+1\right\}\right)$.
\end{center}

Clearly, $C^{i}$ is isomorphic to $C(2,2n+1)$ for $1\leq i\leq
\frac{m}{2}$. By Theorem \ref{mytheorem15}, $C(2,2n+1)\in
\mathfrak{N}$ and there exists an interval $3$-coloring $\alpha$ of
$C(2,2n+1)$. Now we define an edge-coloring $\beta$ of $C(m,2n+1)$.
First we color the edges of $C^{i}$ according to $\alpha$ for $1\leq
i\leq\frac{m}{2}$. Then we color the edges
$v_{j}^{(2i)}v_{j}^{(2i+1)}\in E_{j}$ with color $4$ for $1\leq
i\leq \frac{m}{2}-1, 1\leq j\leq 2n+1$. It is easy to see that
$\beta$ is an interval $4$-coloring of $C(m,2n+1)$. This shows that
$C(m,2n+1)\in \mathfrak{N}$ and $w(C(m,2n+1))\leq 4$. On the other
hand, $w(C(m,2n+1))\geq \Delta(C(m,2n+1))=4$; thus $w(C(m,2n+1))=4$
for even $m$.

Now assume that $m$ is odd.

First we show that $C(3,2n+1)$ has an interval $6$-coloring.

Define an edge-coloring $\gamma$ of $C(3,2n+1)$ as follows:
\begin{description}
\item[(1)]
$\gamma\left(v_{1}^{(1)}v_{1}^{(2)}\right)=6$ and for
$j=2,\ldots,2\left\lfloor\frac{n+1}{2}\right\rfloor$, let
$\gamma\left(v_{j}^{(1)}v_{j}^{(2)}\right)=4$;
\item[(2)]
$\gamma\left(v_{2\left\lfloor\frac{n+1}{2}\right\rfloor+1}^{(1)}v_{2\left\lfloor\frac{n+1}{2}\right\rfloor+1}^{(2)}\right)=2$
and for $j=2\left\lfloor\frac{n+1}{2}\right\rfloor+2,\ldots,2n+1$,
let $\gamma\left(v_{j}^{(1)}v_{j}^{(2)}\right)=3$;
\item[(3)]
$\gamma\left(v_{1}^{(2)}v_{1}^{(3)}\right)=3$ and for
$j=2,\ldots,2\left\lfloor\frac{n+1}{2}\right\rfloor$, let
$\gamma\left(v_{j}^{(2)}v_{j}^{(3)}\right)=2$;
\item[(4)]
for $j=2\left\lfloor\frac{n+1}{2}\right\rfloor+1,\ldots,2n+1$, let
$\gamma\left(v_{j}^{(2)}v_{j}^{(3)}\right)=1$;
\item[(5)]
$j=1,\ldots,\left\lfloor\frac{n+1}{2}\right\rfloor$, let
\begin{center}
$\gamma\left(v_{2j-1}^{(1)}v_{2j}^{(1)}\right)=\gamma\left(v_{2j-1}^{(2)}v_{2j}^{(2)}\right)=5$
and
$\gamma\left(v_{2j}^{(1)}v_{2j+1}^{(1)}\right)=\gamma\left(v_{2j}^{(2)}v_{2j+1}^{(2)}\right)=3$;
\end{center}
\item[(6)]
for $j=\left\lfloor\frac{n+1}{2}\right\rfloor+1,\ldots,n$, let
\begin{center}
$\gamma\left(v_{2j-1}^{(1)}v_{2j}^{(1)}\right)=\gamma\left(v_{2j-1}^{(2)}v_{2j}^{(2)}\right)=4$
and
$\gamma\left(v_{1}^{(1)}v_{2n+1}^{(1)}\right)=\gamma\left(v_{1}^{(2)}v_{2n+1}^{(2)}\right)=4$;
\end{center}
\item[(7)]
for $j=\left\lfloor\frac{n+1}{2}\right\rfloor+1,\ldots,n$, let
$\gamma\left(v_{2j}^{(1)}v_{2j+1}^{(1)}\right)=\gamma\left(v_{2j}^{(2)}v_{2j+1}^{(2)}\right)=2$;
\item[(8)]
for $j=1,\ldots,\left\lfloor\frac{n+1}{2}\right\rfloor$, let
$\gamma\left(v_{2j-1}^{(3)}v_{2j}^{(3)}\right)=1$ and
$\gamma\left(v_{2j}^{(3)}v_{2j+1}^{(3)}\right)=3$;
\item[(9)]
for $j=\left\lfloor\frac{n+1}{2}\right\rfloor+1,\ldots,n$, let
$\gamma\left(v_{2j-1}^{(3)}v_{2j}^{(3)}\right)=2$ and
$\gamma\left(v_{1}^{(3)}v_{2n+1}^{(3)}\right)=2$;
\item[(10)]
for $j=\left\lfloor\frac{n+1}{2}\right\rfloor+1,\ldots,n$, let
$\gamma\left(v_{2j}^{(3)}v_{2j+1}^{(3)}\right)=3$.
\end{description}

It is not difficult to see that $\gamma$ is an interval $6$-coloring
of $C(3,2n+1)$ for which $S(v_{j}^{(3)},\gamma)=[1,3]$ when $1\leq
j\leq 2n+1$.

Next we define an edge-coloring $\phi$ of $C(m,2n+1)$ as follows:
first we color the edges of the subgraph $C(3,2n+1)$ of $C(m,2n+1)$
according to $\gamma$. Secondly, we color the edges of the remaining
subgraph $C(m-3,2n+1)$ of $C(m,2n+1)$ according to $\beta$, and
finally, we color the edges $v_{j}^{(3)}v_{j}^{(4)}\in E_{j}$ with
color $4$ for $1\leq j\leq 2n+1$. It is easy to see that $\phi$ is
an interval $6$-coloring of $C(m,2n+1)$. This shows that
$C(m,2n+1)\in \mathfrak{N}$ and $w(C(m,2n+1))\leq 6$.

Now we prove that $w(C(m,2n+1))\geq 6$ for odd $m$.

Let $\psi$ be an interval $w(C(m,2n+1))$-coloring of $C(m,2n+1)$ and
$w(C(m,2n+1))\leq 5$. Consider the set
$S\left(v_{j}^{(i)},\psi\right)$ for $1\leq i\leq m, 1\leq j\leq
2n+1$. It is easy to see that if $d\left(v_{j}^{(i)}\right)=3$, then
$1\leq \min S\left(v_{j}^{(i)},\psi\right)\leq 3$, and if
$d\left(v_{j}^{(i)}\right)=4$, then $1\leq \min
S\left(v_{j}^{(i)},\psi\right)\leq 2$. Hence, $3\in
S\left(v_{j}^{(i)},\psi\right)$ for $1\leq i\leq m, 1\leq j\leq
2n+1$, but this implies that the edges with color $3$ form a perfect
matching in $C(m,2n+1)$, which contradicts the fact that $C(m,2n+1)$
does not have one. Thus, $w(C(m,2n+1))=6$ for odd $m$. ~$\square$
\end{proof}

Before we derive lower bounds for $W(C(2m,2n))$ and $W(C(2m,2n+1))$,
let us note that Lemma \ref{mylemma2}, Theorems \ref{mytheorem11}
and \ref{mytheorem16} imply the following:

\begin{corollary}
\label{mycorollary4} If $G\square H$ is planar and both factors have
at least $3$ vertices, then $G\square H\in \mathfrak{N}$ and
$w(G\square H)\leq 6$.
\end{corollary}

\begin{theorem}
\label{mytheorem17} If $m\in\mathbb{N},n\geq 2$, then
$W(C(2m,2n))\geq 4m+2n-2$, and if $m,n\in\mathbb{N}$, then
$W(C(2m,2n+1))\geq 4m+2n-1$.
\end{theorem}
\begin{proof} For the proof of the theorem, it suffices to construct
edge-colorings that satisfies the specified conditions. Let
$V(C(2m,2n))=\left\{v_{j}^{(i)}\colon\,1\leq i\leq 2m,1\leq j\leq
2n\right\}$ and $V(C(2m,2n+1))=\left\{u_{j}^{(i)}\colon\,1\leq i\leq
2m,1\leq j\leq 2n+1\right\}$. Also, let
$E(C(2m,2n))=\bigcup_{i=1}^{2m}E^{i}\cup \bigcup_{j=1}^{2n}E_{j}$
and $E(C(2m,2n+1))=\bigcup_{i=1}^{2m}\overline{E}^{i}\cup
\bigcup_{j=1}^{2n+1}\overline{E}_{j}$ , where
\begin{center}
$E^{i}=\left\{v_{j}^{(i)}v_{j+1}^{(i)}\colon\,1\leq j\leq
2n-1,\right\}\cup \left\{v_{1}^{(i)}v_{2n}^{(i)}\right\}$,
$E_{j}=\left\{v_{j}^{(i)}v_{j}^{(i+1)}\colon\,1\leq i\leq
2m-1\right\}$ and
$\overline{E}^{i}=\left\{u_{j}^{(i)}u_{j+1}^{(i)}\colon\,1\leq j\leq
2n,\right\}\cup \left\{u_{1}^{(i)}u_{2n+1}^{(i)}\right\}$,
$\overline{E}_{j}=\left\{u_{j}^{(i)}u_{j}^{(i+1)}\colon\,1\leq i\leq
2m-1\right\}$.
\end{center}

First we construct an interval $(4m+2n-2)$-coloring of $C(2m,2n)$
when $m\in\mathbb{N},n\geq 2$.

Define an edge-coloring $\alpha$ of $C(2m,2n)$ as follows:\\
\begin{description}
\item[(1)] for $i=1,\ldots,m$, $j=1,\ldots,n$, let
\begin{center}
$\alpha\left(v_{j}^{(2i-1)}v_{j+1}^{(2i-1)}\right)=\alpha\left(v_{j}^{(2i)}v_{j+1}^{(2i)}\right)=4i+2j-4$;
\end{center}
\item[(2)] for $i=1,\ldots,m$, $j=n+1,\ldots,2n-1$, let
\begin{center}
$\alpha\left(v_{j}^{(2i-1)}v_{j+1}^{(2i-1)}\right)=\alpha\left(v_{j}^{(2i)}v_{j+1}^{(2i)}\right)=4i-2j+4n-1$;
\end{center}
\item[(3)] for $i=1,\ldots,m$, let
\begin{center}
$\alpha\left(v_{1}^{(2i-1)}v_{2n}^{(2i-1)}\right)=\alpha\left(v_{1}^{(2i)}v_{2n}^{(2i)}\right)=4i-1$;
\end{center}
\item[(4)] for $i=1,\ldots,m$, $j=1,\ldots,n$, let
\begin{center}
$\alpha\left(v_{j}^{(2i-1)}v_{j}^{(2i)}\right)=4i+2j-5$;
\end{center}
\item[(5)] for $i=1,\ldots,m$, $j=n+1,\ldots,2n$, let
\begin{center}
$\alpha\left(v_{j}^{(2i-1)}v_{j}^{(2i)}\right)=4i-2j+4n$;
\end{center}
\item[(6)] for $i=1,\ldots,m-1$, $j=2,\ldots,n+1$, let
\begin{center}
$\alpha\left(v_{j}^{(2i)}v_{j}^{(2i+1)}\right)=4i+2j-3$;
\end{center}
\item[(7)] for $i=1,\ldots,m-1$, $j=n+2,\ldots,2n$, let
\begin{center}
$\alpha\left(v_{j}^{(2i)}v_{j}^{(2i+1)}\right)=4i-2j+4n+2$;
\end{center}
\item[(8)] for $i=1,\ldots,m-1$, let
\begin{center}
$\alpha\left(v_{1}^{(2i)}v_{1}^{(2i+1)}\right)=4i$.
\end{center}
\end{description}

Next we construct an interval $(4m+2n-1)$-coloring of $C(2m,2n+1)$
when $m,n\in\mathbb{N}$.

Define an edge-coloring $\beta$ of $C(2m,2n+1)$ as follows:\\
\begin{description}
\item[(1)] for $i=1,\ldots,m$, $j=1,\ldots,n+1$, let
\begin{center}
$\beta\left(u_{j}^{(2i-1)}u_{j+1}^{(2i-1)}\right)=\beta\left(u_{j}^{(2i)}u_{j+1}^{(2i)}\right)=4i+2j-4$;
\end{center}
\item[(2)] for $i=1,\ldots,m$, $j=n+2,\ldots,2n$, let
\begin{center}
$\beta\left(u_{j}^{(2i-1)}u_{j+1}^{(2i-1)}\right)=\beta\left(u_{j}^{(2i)}u_{j+1}^{(2i)}\right)=4i-2j+4n+1$;
\end{center}
\item[(3)] for $i=1,\ldots,m$, let
\begin{center}
$\beta\left(u_{1}^{(2i-1)}u_{2n+1}^{(2i-1)}\right)=\beta\left(u_{1}^{(2i)}u_{2n+1}^{(2i)}\right)=4i-1$;
\end{center}
\item[(4)] for $i=1,\ldots,m$, $j=1,\ldots,n+2$, let
\begin{center}
$\beta\left(u_{j}^{(2i-1)}u_{j}^{(2i)}\right)=4i+2j-5$;
\end{center}
\item[(5)] for $i=1,\ldots,m$, $j=n+3,\ldots,2n+1$, let
\begin{center}
$\beta\left(u_{j}^{(2i-1)}u_{j}^{(2i)}\right)=4i-2j+4n+2$;
\end{center}
\item[(6)] for $i=1,\ldots,m-1$, $j=2,\ldots,n+1$, let
\begin{center}
$\beta\left(u_{j}^{(2i)}u_{j}^{(2i+1)}\right)=4i+2j-3$;
\end{center}
\item[(7)] for $i=1,\ldots,m-1$, $j=n+2,\ldots,2n+1$, let
\begin{center}
$\beta\left(u_{j}^{(2i)}u_{j}^{(2i+1)}\right)=4i-2j+4n+4$;
\end{center}
\item[(8)] for $i=1,\ldots,m-1$, let
\begin{center}
$\beta\left(u_{1}^{(2i)}u_{1}^{(2i+1)}\right)=4i$.
\end{center}
\end{description}

It is straightforward to check that $\alpha$ is an interval
$(4m+2n-2)$-coloring of $C(2m,2n)$ when $m\in\mathbb{N},n\geq 2$,
and $\beta$ is an interval $(4m+2n-1)$-coloring of $C(2m,2n+1)$ when
$m,n\in\mathbb{N}$.~$\square$
\end{proof}

Note that the lower bound in Theorem \ref{mytheorem17} is not so far
from the upper bound for $W\left(C(m,n)\right)$. Indeed, since
$C(2m,2n)$ is bipartite, $3\leq \Delta\left(C(2m,2n)\right)\leq 4$
and $\mathrm{diam}\left(C(2m,2n)\right)=2m+n-1$, by Theorem
\ref{mytheorem3}, we have $W\left(C(2m,2n)\right)\leq 3(2m+n-1)+1$.
Similarly, since $3\leq \Delta\left(C(2m,2n+1)\right)\leq 4$ and
$\mathrm{diam}\left(C(2m,2n+1)\right)=2m+n-1$, by Theorem
\ref{mytheorem2}, we have $W\left(C(2m,2n+1)\right)\leq 3(2m+n)+1$.
Next we would like to compare obtained lower bounds for $W(C(m,n))$.
If $m$ is even and $m<n$, then the lower bound in Theorem
\ref{mytheorem17} is better than in Theorem \ref{mytheorem12}, if
$m$ is even and $m>n$, then the lower bound in Theorem
\ref{mytheorem12} is better than in Theorem \ref{mytheorem17}, and
if $m$ is even and $m=n$, then we obtain the
same lower bound in both theorems.\\

In the following we consider tori. In \cite{b22}, Petrosyan proved
that the torus $T(m,n)\in\mathfrak{N}$ if and only if $mn$ is even.
Since $T(m,n)$ is $4$-regular, by Theorem \ref{mytheorem1}, we
obtain that $w(T(m,n))=4$ when $mn$ is even. Now we derive a new
lower bound for $W(T(m,n))$ when $mn$ is even.

\begin{theorem}
\label{mytheorem18} For any $m,n \geq 2$, we have $W(T(2m,2n))\geq
\max\{3m+n+2,3n+m+2\}$, and for any $m\geq 2$, $n\in\mathbb{N}$, we
have
\begin{center}
$W\left(T(2m,2n+1)\right)\geq \left\{
\begin{tabular}{ll}
$2m+2n+2$, & if $m$ is odd,\\
$2m+2n+3$, & if $m$ is even.\\
\end{tabular}%
\right.$
\end{center}
\end{theorem}
\begin{proof} First note that the lower bound for $W(T(2m,2n))$ $(m,n \geq
2)$ follows from Theorem \ref{mytheorem10}. For the proof of a
second part of the theorem, it suffices to construct an
edge-coloring of $T(2m,2n+1)$ that satisfies the specified
conditions.

Let $V(T(2m,2n+1))=\left\{v_{j}^{(i)}\colon\,1\leq i\leq 2m,1\leq
j\leq 2n+1\right\}$ and $E(T(2m,2n+1))=\bigcup_{i=1}^{2m}E^{i}\cup
\bigcup_{j=1}^{2n+1}E_{j}$, where
\begin{center}
$E^{i}=\left\{v_{j}^{(i)}v_{j+1}^{(i)}\colon\,1\leq j\leq
2n\right\}\cup \left\{v_{1}^{(i)}v_{2n+1}^{(i)}\right\}$,
$E_{j}=\left\{v_{j}^{(i)}v_{j}^{(i+1)}\colon\,1\leq i\leq
2m-1\right\}\cup \left\{v_{j}^{(1)}v_{j}^{(2m)}\right\}$.
\end{center}

Define an edge-coloring $\alpha$ of $T(2m,2n+1)$ as follows:\\
\begin{description}
\item[(1)] for $j=1,\ldots,n+1$, let
\begin{center}
$\alpha\left(v_{j}^{(1)}v_{j+1}^{(1)}\right)=\alpha\left(v_{j}^{(2m)}v_{j+1}^{(2m)}\right)=2j$;
\end{center}
\item[(2)] for $j=n+2,\ldots,2n$, let
\begin{center}
$\alpha\left(v_{j}^{(1)}v_{j+1}^{(1)}\right)=\alpha\left(v_{j}^{(2m)}v_{j+1}^{(2m)}\right)=2(2n+1-j)+3$
and\\
$\alpha\left(v_{1}^{(1)}v_{2n+1}^{(1)}\right)=\alpha\left(v_{1}^{(2m)}v_{2n+1}^{(2m)}\right)=3$;
\end{center}
\item[(3)] for $j=1,\ldots,n+2$, let
\begin{center}
$\alpha\left(v_{j}^{(1)}v_{j}^{(2m)}\right)=2j-1$;
\end{center}
\item[(4)] for $j=n+3,\ldots,2n+1$, let
\begin{center}
$\alpha\left(v_{j}^{(1)}v_{j}^{(2m)}\right)=2(2n+3-j)$;
\end{center}
\item[(5)] for $i=1,\ldots,\left\lfloor\frac{m}{2}\right\rfloor$, $j=1,\ldots,n+1$, let
\begin{center}
$\alpha\left(v_{j}^{(2i)}v_{j+1}^{(2i)}\right)=\alpha\left(v_{j}^{(2i+1)}v_{j+1}^{(2i+1)}\right)=
\alpha\left(v_{j}^{(2m-2i)}v_{j+1}^{(2m-2i)}\right)=\alpha\left(v_{j}^{(2m-2i+1)}v_{j+1}^{(2m-2i+1)}\right)=4i+2j$;
\end{center}
\item[(6)] for $i=1,\ldots,\left\lfloor\frac{m}{2}\right\rfloor$, $j=n+2,\ldots,2n$, let
\begin{center}
$\alpha\left(v_{j}^{(2i)}v_{j+1}^{(2i)}\right)=\alpha\left(v_{j}^{(2i+1)}v_{j+1}^{(2i+1)}\right)=
\alpha\left(v_{j}^{(2m-2i)}v_{j+1}^{(2m-2i)}\right)=\alpha\left(v_{j}^{(2m-2i+1)}v_{j+1}^{(2m-2i+1)}\right)=4i+2(2n+1-j)+3$
and\\
$\alpha\left(v_{1}^{(2i)}v_{2n+1}^{(2i)}\right)=\alpha\left(v_{1}^{(2i+1)}v_{2n+1}^{(2i+1)}\right)=\alpha\left(v_{1}^{(2m-2i)}v_{2n+1}^{(2m-2i)}\right)=\alpha\left(v_{1}^{(2m-2i+1)}v_{2n+1}^{(2m-2i+1)}\right)=4i+3$;
\end{center}
\item[(7)] for $i=1,\ldots,\left\lceil\frac{m}{2}\right\rceil$, $j=2,\ldots,n+1$, let
\begin{center}
$\alpha\left(v_{j}^{(2i-1)}v_{j}^{(2i)}\right)=\alpha\left(v_{j}^{(2m-2i+1)}v_{j}^{(2m-2i+2)}\right)=4i+2j-3$;
\end{center}
\item[(8)] for $i=1,\ldots,\left\lceil\frac{m}{2}\right\rceil$, $j=n+2,\ldots,2n+1$, let
\begin{center}
$\alpha\left(v_{j}^{(2i-1)}v_{j}^{(2i)}\right)=\alpha\left(v_{j}^{(2m-2i+1)}v_{j}^{(2m-2i+2)}\right)=4(n+1+i)-2j$;
\end{center}
\item[(9)] for $i=1,\ldots,\left\lceil\frac{m}{2}\right\rceil$, let
\begin{center}
$\alpha\left(v_{1}^{(2i-1)}v_{1}^{(2i)}\right)=\alpha\left(v_{1}^{(2m-2i+1)}v_{1}^{(2m-2i+2)}\right)=4i$;
\end{center}
\item[(10)] for $i=1,\ldots,\left\lfloor\frac{m}{2}\right\rfloor$, $j=1,\ldots,n+2$, let
\begin{center}
$\alpha\left(v_{j}^{(2i)}v_{j}^{(2i+1)}\right)=\alpha\left(v_{j}^{(2m-2i)}v_{j}^{(2m-2i+1)}\right)=4i+2j-1$;
\end{center}
\item[(11)] for $i=1,\ldots,\left\lfloor\frac{m}{2}\right\rfloor$, $j=n+3,\ldots,2n+1$, let
\begin{center}
$\alpha\left(v_{j}^{(2i)}v_{j}^{(2i+1)}\right)=\alpha\left(v_{j}^{(2m-2i)}v_{j}^{(2m-2i+1)}\right)=4i+2(2n+3-j)$.
\end{center}
\end{description}
It is not difficult to check that $\alpha$ is an interval
$(2m+2n+3)$-coloring of $T(2m,2n+1)$ when $m$ is even, and $\alpha$
is an interval $(2m+2n+2)$-coloring of $T(2m,2n+1)$ when $m$ is odd.
~$\square$
\end{proof}

From Theorems \ref{mytheorem1}, \ref{mytheorem11} and
\ref{mytheorem18}, we have:

\begin{corollary}
\label{mycorollary5} If $G=T(2m,2n)$ $(m,n\geq 2)$ and $4\leq t\leq
\max\{3m+n+2,3n+m+2\}$, then $G$ has an interval $t$-coloring. Also,
If $H=T(2m,2n+1)$ $(m\geq 2, n\in\mathbb{N})$, $m$ is odd and $4\leq
t\leq 2m+2n+2$, then $H$ has an interval $t$-coloring, and if
$H=T(2m,2n+1)$ $(m\geq 2, n\in\mathbb{N})$, $m$ is even and $4\leq
t\leq 2m+2n+3$, then $H$ has an interval $t$-coloring.
\end{corollary}

Let us note that the lower bound in Theorem \ref{mytheorem18} is not
so far from the upper bound for $W\left(T(m,n)\right)$. Indeed,
since $T(2m,2n)$ is bipartite, $\Delta\left(T(2m,2n)\right)=4$ and
$\mathrm{diam}\left(C(2m,2n)\right)=m+n$, by Theorem
\ref{mytheorem3}, we have $W\left(T(2m,2n)\right)\leq 3(m+n)+1$.
Similarly, since $\Delta\left(T(2m,2n+1)\right)=4$ and
$\mathrm{diam}\left(T(2m,2n+1)\right)=m+n$, by Theorem
\ref{mytheorem2}, we have $W\left(T(2m,2n+1)\right)\leq 3(m+n+1)+1$.
\bigskip

\section{$N$-dimensional cubes}\

It is well-known that the $n$-dimensional cube $Q_{n}$ is the
Cartesian product of $n$ copies of $K_{2}$. In \cite{b21}, Petrosyan
investigated interval colorings of $n$-dimensional cubes and proved
that $w\left(Q_{n}\right)=n$ and
$W\left(Q_{n}\right)\geq\frac{n(n+1)}{2}$ for any $n\in\mathbb{N}$.
In the same paper he also conjectured that $W\left(Q_{n}\right)
=\frac{n(n+1)}{2}$ for any $n\in\mathbb{N}$. Here, we prove this
conjecture.

Let $e,e^{\prime}\in E(Q_{n})$ and $e=u_{1}u_{2}$,
$e^{\prime}=v_{1}v_{2}$. The distance between two edges $e$ and
$e^{\prime}$ in $Q_{n}$, we define as follows:

\begin{center}
$d(e,e^{\prime})=\min_{1\leq i\leq 2,1\leq j\leq 2}
d\left(u_{i},v_{j}\right)$.
\end{center}

Let $\alpha$ be an interval $t$-coloring of $Q_{n}$.

For any $e,e^{\prime}\in E(Q_{n})$, define
$\mathrm{sp}_{\alpha}\left(e,e^{\prime}\right)$ as follows:
\begin{center}
$\mathrm{sp}_{\alpha}\left(e,e^{\prime}\right)=\left\vert
\alpha(e)-\alpha(e^{\prime})\right\vert$.
\end{center}

For any $k, 0\leq k\leq n-1$, define $\mathrm{sp}_{\alpha,k}$ as
follows:

\begin{center}
$\mathrm{sp}_{\alpha,k}=\max
\left\{\mathrm{sp}_{\alpha}\left(e,e^{\prime}\right)\colon\,e,e^{\prime}\in
E(Q_{n})~and~d(e,e^{\prime})=k\right\}$.
\end{center}

First we need the following simple lemma.

\begin{lemma}
\label{mylemma3} For any pair of vertices $u,v\in V(Q_{n})$ with
$d(u,v)=k$, there exist $v_{1},v_{2},\ldots,v_{k}$ ($v_{i}\neq
v_{j}$ when $i\neq j$) vertices such that $d(u,v_{i})=k-1$ and
$vv_{i}\in E(Q_{n})$ for $i=1,\ldots,k$.
\end{lemma}

Now let $\alpha$ be an interval $W(Q_{n})$-coloring of $Q_{n}$.
Clearly, $\mathrm{sp}_{\alpha,0}=n-1$.

\begin{theorem}
\label{mytheorem19} If $1\leq k\leq n-1$, then
$\mathrm{sp}_{\alpha,k}\leq \mathrm{sp}_{\alpha,k-1}+n-k$.
\end{theorem}
\begin{proof}
Let $e,e^{\prime}\in E(Q_{n})$ be any two edges of $Q_{n}$ with
$d(e,e^{\prime})=k$. Without loss of generality, we may assume that
$\alpha(e)\geq \alpha(e^{\prime})$. Since $d(e,e^{\prime})=k$, there
exist $u$ and $v$ vertices such that $u\in e$ and $v\in e^{\prime}$
and $d(u,v)=k$. By Lemma \ref{mylemma3}, we have that there are
$v_{1},v_{2},\ldots,v_{k}$ ($v_{i}\neq v_{j}$ when $i\neq j$)
vertices such that $d(u,v_{i})=k-1$ and $vv_{i}\in E(Q_{n})$ for
$i=1,\ldots,k$. Since $Q_{n}$ is $n$-regular, we have
\begin{center}
$\min_{1\leq i\leq k} \alpha(v_{i}v)\leq \alpha(e^{\prime})+n-k$.
(*)
\end{center}

Let $\alpha(e^{\prime\prime})=\min_{1\leq i\leq k} \alpha(v_{i}v)$.
By (*), we obtain
\begin{center}
$\alpha(e^{\prime})\geq \alpha(e^{\prime\prime})-(n-k)$ and
$d(e,e^{\prime\prime})=k-1$.
\end{center}

Thus,
\begin{center}
$\mathrm{sp}_{\alpha}\left(e,e^{\prime}\right)=\left\vert
\alpha(e)-\alpha(e^{\prime})\right\vert \leq \left\vert
\alpha(e)-\alpha(e^{\prime\prime})+n-k\right\vert\leq \left\vert
\alpha(e)-\alpha(e^{\prime\prime})\right\vert+n-k\leq
\mathrm{sp}_{\alpha,k-1}+n-k$.
\end{center}

Since $e$ and $e^{\prime}$ were arbitrary edges with
$d(e,e^{\prime})=k$, we obtain $\mathrm{sp}_{\alpha,k}\leq
\mathrm{sp}_{\alpha,k-1}+n-k$. ~$\square$
\end{proof}

\begin{corollary}
\label{mycorollary6} $\mathrm{sp}_{\alpha,n-1}\leq \frac{n\left(
n+1\right)}{2}-1$.
\end{corollary}
\begin{proof} By Theorem \ref{mytheorem19}, we have
\begin{center}
$\mathrm{sp}_{\alpha,n-1}\leq
\mathrm{sp}_{\alpha,0}+n-1+n-2+\ldots+1 = \frac{n\left(
n+1\right)}{2}-1$.
\end{center} ~$\square$
\end{proof}

\begin{corollary}
\label{mycorollary7} $W\left(Q_{n}\right) \leq \frac{n\left(
n+1\right) }{2}$ for any $n\in\mathbb{N}$.
\end{corollary}
\begin{proof}
Clearly, for any $e,e^{\prime}\in E(Q_{n})$, we have
$d(e,e^{\prime})\leq n-1$. Thus, by Corollary \ref{mycorollary6}, we
get $W\left(Q_{n}\right)\leq \frac{n\left( n+1\right)}{2}$.
~$\square$
\end{proof}

By Theorem \ref{mytheorem6} and Corollary \ref{mycorollary7}, we
obtain $W\left(Q_{n}\right)=\frac{n\left(n+1\right)}{2}$ for any
$n\in\mathbb{N}$. Moreover, by Theorem \ref{mytheorem1}, we have
that $Q_{n}$ has an interval $t$-coloring if and only if $n\leq
t\leq \frac{n\left(n+1\right)}{2}$.

\end{document}